\documentclass[12pt]{amsart}
\usepackage{amsmath, amssymb}

\theoremstyle{plain}
\newtheorem{theorem}{Theorem}[section]
\newtheorem{proposition}[theorem]{Proposition}

\newtheorem{lemma}[theorem]{Lemma}

\theoremstyle{definition}
\newtheorem{definition}[theorem]{Definition}

\theoremstyle{remark}
\newtheorem*{remark}{Remark}
\theoremstyle{remark}

\numberwithin{equation}{section}

\DeclareMathOperator{\LCD}{LCD}

\DeclareMathOperator{\Spread}{Spread}

\def \N {\mathbb{N}}
\def \R {\mathbb{R}}

\def \Z {\mathbb{Z}}
\def \E {\mathbb{E}}
\def \P {\mathbb{P}}

\def \DD {\mathcal{D}}
\def \EE {\mathcal{E}}
\def \NN {\mathcal{N}}

\def \LL {\mathcal{L}}
\def \MM {\mathcal{M}}

\def \a {\alpha}
\def \b {\beta}
\def \g {\gamma}
\def \e {\varepsilon}

\def \d {\delta}

\def \l {\lambda}

\def \s {\sigma}
\def \t {\tau}

\def \< {\langle}
\def \> {\rangle}
\def \^ {\widehat}

\def \dist {{\rm dist}}

\def \Span {{\rm span}}

\def \vol {{\rm vol}}

\def \supp {{\rm supp}}

\def \Comp {{\mathit{Comp}}}
\def \Incomp {{\mathit{Incomp}}}

\newcommand{\pr}[2]{\langle {#1} , {#2} \rangle}
\newcommand{\norm}[1]{\left \| #1 \right \|}
\def \etc {,\ldots,}

\begin{document}

\title[]{The smallest singular value of a random rectangular matrix}

\author{Mark Rudelson
  \and Roman Vershynin}

\thanks{M.R. was supported by NSF DMS grants 0556151 and 0652684.
  R.V. was supported by the Alfred P.~Sloan Foundation
  and by NSF DMS grants 0401032 and 0652617.}

\address{Department of Mathematics,
   University of Missouri,
   Columbia, MO 65211, USA}
\email{rudelson@math.missouri.edu}

\address{Department of Mathematics,
   University of Michigan,
   Ann Arbor, MI 48109, USA}
\email{romanv@umich.edu}

\date{\today}

\begin{abstract}
  We prove an optimal estimate of the smallest singular value
  of a random subgaussian matrix, valid for all  dimensions.
  For an $N \times n$ matrix $A$ with
  independent and identically distributed subgaussian entries,
  the smallest singular value of $A$ is at least of the order
  $\sqrt{N} - \sqrt{n-1}$ with high probability.
  A sharp estimate on the probability is also obtained.
\end{abstract}

\maketitle


\section{Introduction}

\subsection{Singular values of subgaussian matrices}
Extreme singular values of random matrices has been
of considerable interest in mathematical physics, geometric functional analysis,
numerical analysis and other fields.
Consider an $N \times n$ real matrix $A$ with $N \ge n$.
The {\em singular values} $s_k(A)$ of $A$ are the eigenvalues of $|A| =
\sqrt{A^tA}$ arranged in nonincreasing order. Of particular
significance are the largest and the smallest singular values
\begin{equation}                                \label{s1 sn}
  s_1(A) = \sup_{x:\; \|x\|_2 = 1} \|Ax\|_2, \qquad s_n(A) =
  \inf_{x:\; \|x\|_2 = 1} \|Ax\|_2.
\end{equation}

A natural matrix model is given by matrices whose entries are
independent real random variables with certain moment assumptions.
In this paper, we shall consider {\em subgaussian random variables}
$\xi$ -- those whose tails are dominated by that of the standard
normal random variable. Namely, a random variable $\xi$ is called
subgaussian if there exists $B>0$ such that
\begin{equation}   \label{subgaussian}
  \P (|\xi| > t) \le 2 \exp(-t^2/B^2) \qquad \text{for all $t > 0$}.
\end{equation}
The minimal $B$ in this inequality is called the {\em subgaussian
moment} of $\xi$. Inequality \eqref{subgaussian} is often
equivalently formulated as the moment condition
\begin{equation}    \label{subgaussian moments}
   (\E |\xi|^p )^{1/p} \le C B \sqrt{p} \qquad \text{for all $p \ge 1$},
\end{equation}
where $C$ is an absolute constant.
The class of subgaussian random variables includes many random variables
that arise naturally in applications, such as normal, symmetric $\pm 1$
and general bounded random variables.

In this paper, we study $N \times n$ real random matrices $A$
whose entries are independent and identically distributed mean zero subgaussian random
variables. The asymptotic behavior of the extreme singular values of $A$ is well understood.
If the entries have unit variance and the dimension $n$ grows to infinity while the aspect ratio $n/N$
converges to a constant $\l \in (0,1)$, then
\[
  \frac{s_1(A)}{\sqrt{N}} \to 1+ \sqrt{\l}, \qquad
  \frac{s_n(A)}{\sqrt{N}} \to 1- \sqrt{\l}
\]
almost surely. This result was proved in \cite{Si} for Gaussian
matrices, and in \cite{BY} for matrices with independent and identically
distributed entries with finite fourth moment. In other words, we have
asymptotically
\begin{equation}                                \label{limit}
  s_1(A) \sim \sqrt{N} + \sqrt{n}, \qquad
  s_n(A) \sim \sqrt{N} - \sqrt{n}.
\end{equation}

Considerable efforts were made recently to establish non-asymptotic
estimates similar to \eqref{limit}, which would hold for arbitrary
fixed dimensions $N$ and $n$; see the survey \cite{L} on the largest singular value,
and the discussion below on the smallest singular value.

Estimates in fixed dimensions are essential for many problems of geometric functional
analysis and computer science. Most often needed are upper bounds on the largest
singular value and lower bounds on the smallest
singular value, which together yield that $A$ acts as a nice
isomorphic embedding of $\R^n$ into $\R^N$.
Such bounds are often satisfactory even if they are known to hold up
to a constant factor independent of the dimension.

The largest singular value is relatively easy to bound above, up to a constant factor.
Indeed, a standard covering argument shows that
$s_1(A)$ is at most of the optimal order $\sqrt{N}$ for all
fixed dimensions, see Proposition~\ref{norm} below.
The smallest singular value is significantly harder to control.
The efforts to prove optimal bounds on $s_n(A)$ have a long history, which
we shall now outline.

\subsection{Tall matrices}

A result of \cite{BDGJN} provides an optimal bound for
tall matrices, those with aspect ratio $\l = n/N$ satisfies $\l < \l_0$
for some sufficiently small constant $\l_0 > 0$.
Recalling \eqref{limit},
one should expect that tall matrices satisfy
\begin{equation}                    \label{tall}
  s_n(A) \ge c \sqrt{N} \qquad \text{with high probability}.
\end{equation}
It was indeed proved in \cite{BDGJN} that for tall $\pm 1$ matrices one has
\begin{equation}                    \label{tall probability}
  \P \big( s_n(A) \le c \sqrt{N} \big)
  \le e^{-c N}
\end{equation}
where $\l_0 > 0$ and $c > 0$ are absolute constants.

\subsection{Almost square matrices}

As we move toward square matrices, thus making the aspect ratio $\l = n/N$
arbitrarily close to $1$, the problem of estimating the smallest singular value becomes
harder. One still expects \eqref{tall} to be true as long as $\l < 1$ is any constant.
Indeed, this was proved in \cite{LPRT} for arbitrary
aspect ratios $\l < 1 - c/\log n$ and for general random matrices
with independent subgaussian entries. One has
\begin{equation}  \label{constant proportion}
  \P \big( s_n(A) \le c_\l \sqrt{N} \big)
  \le e^{-c N},
\end{equation}
where $c_\l > 0$ depends only on $\l$ and the maximal subgaussian moment
of the entries.

In subsequent work \cite{AFMS}, the dependence of $c_\l$ on the
aspect ratio in \eqref{constant proportion} was improved for random
$\pm 1$ matrices; however the probability estimate there was weaker
than in \eqref{constant proportion}.   An estimate for subgaussian
random matrices of all dimensions was obtained in \cite{R2}. For any
$\e \ge CN^{-1/2}$, it was shown that
  \[
    \P \big( s_n(A) \le \e (1-\l)(\sqrt{N} - \sqrt{n}) \big)
    \le (C\e)^{N-n} + e^{-cN}.
  \]
However, because of the factor $(1-\l)$, this estimate is suboptimal
and does not correspond to the expected asymptotic behavior \eqref{limit}.

\subsection{Square matrices}

The extreme case for the problem of estimating the singular value is
for the square matrices, where $N=n$. Asymptotic \eqref{limit} is
useless for square matrices. However, for ``almost'' square
matrices, those with constant defect $N - n = O(1)$, the quantity
$\sqrt{N} - \sqrt{n}$ is of order $1/\sqrt{N}$, so asymptotics
\eqref{limit} heuristically suggests that these matrices should
satisfy
\begin{equation}                    \label{square}
  s_n(A) \ge  \frac{c}{\sqrt{N}}
  \qquad \text{with high probability}.
\end{equation}
This conjecture was proved recently in \cite{RV} for all square subgaussian
matrices:
\begin{equation}                    \label{square probability}
  \P \Big( s_n(A) \le \frac{\e}{\sqrt{N}} \Big)
  \le C\e +  e^{-cN}.
\end{equation}

\subsection{New result: bridging all classes of matrices}

In this paper, we prove the conjectural bound for $s_n(A)$ valid for all
subgaussian matrices in all fixed dimensions $N, n$. The bound is optimal
for matrices with all aspect ratios we encountered above.

\begin{theorem}             \label{t: subgaussian}
  Let $A$ be an $N \times n$ random matrix, $N \ge n$, whose elements
  are independent copies of a mean zero subgaussian random variable with unit variance.
  Then, for every $\e > 0$, we have
  \begin{equation}                                  \label{eq rectangular subgaussian}
    \P \Big( s_n(A) \le \e \big (\sqrt{N} - \sqrt{n-1} \big ) \Big)
    \le (C \e)^{N-n+1} + e^{-cN}
  \end{equation}
  where $C, c > 0$ depend (polynomially)
  only on the subgaussian moment $B$.
\end{theorem}

For {\em tall matrices}, Theorem \ref{t: subgaussian} clearly amounts to the
known estimates \eqref{tall}, \eqref{tall probability}.
For {\em square matrices} ($N=n$), the quantity
$\sqrt{N} - \sqrt{N-1}$ is of order $1/\sqrt{N}$, so
Theorem \ref{t: subgaussian} amounts to the known estimates
\eqref{square}, \eqref{square probability}.
Finally, for matrices that are arbitrarily {\em close to square},
Theorem \ref{t: subgaussian} yields the new optimal estimate
\begin{equation}                    \label{no limit}
  s_n(A) \ge c( \sqrt{N} - \sqrt{n} )
  \qquad \text{with high probability}.
\end{equation}
This is a version of the asymptotics \eqref{limit}, now valid for all
fixed dimensions. This bound was explicitly conjectured e.g. in \cite{V}.

\medskip

Theorem~\ref{t: subgaussian} seems to be new even for Gaussian matrices.
Some early progress was made by Edelman \cite{E 88}
and Szarek \cite{Sz} who in particular proved \eqref{square probability} for Gaussian matrices,
see also the subsequent work by Edelman and Sutton \cite{ES}.
Gordon's inequality \cite{G} can be used to prove that, for Gaussian matrices,
$\E s_n(A) \ge \sqrt{N} - \sqrt{n}$, see Theorem II.13 in \cite{DS}.
One can further use the concentration of measure inequality on the
Euclidean sphere to estimate the probability as
$$
\P \big( s_n(A) \le \sqrt{N} - \sqrt{n} - t \big)
    \le e^{-t^2/2}, \qquad t > 0.
$$
However, this bound is not optimal, and it becomes useless
for matrices that are close to square, when $N - n = o(\sqrt{n})$.

\medskip

The form of estimate \eqref{eq rectangular subgaussian} may be expected if one recalls
the classical $\e$-net argument, which underlies many proofs in geometric
functional analysis. By \eqref{s1 sn}, we are looking for a lower bound on
$\|Ax\|$ that would hold uniformly for all vectors $x$ on the unit Euclidean sphere
$S^{n-1}$. For every fixed $x \in S^{n-1}$, the quantity $\|Ax\|_2^2$ is the sum of $N$
independent random variables (the squares of the coordinates of $Ax$).
Therefore, the deviation inequalities make us to expect that $\|Ax\|_2$ is of the order
$\sqrt{N}$ with probability exponential in $N$, i.e. $1-e^{-cN}$. We can run this
argument separately for each vector $x$ in a small net $\NN$ of the sphere $S^{n-1}$, and
then take the union bound to make the estimate uniform over $x \in \NN$.
It is known how to choose a net $\NN$ of cardinality exponential in the dimension $n-1$ of the sphere,
i.e. $|\NN| \le e^{C(n-1)}$. Therefore, with probability $1 - e^{C(n-1)} e^{-cN}$, we have
a good lower bound on $\|Ax\|_2 \sim \sqrt{N}$ for all vectors $x$ in the net $\NN$. Finally, one transfers
this estimate from the net to the whole sphere $S^{n-1}$ by approximation.

The problem with this argument is that the constants $C$ and $c$ are not the same.
Therefore, our estimate on the probability $1 - e^{C(n-1)} e^{-cN}$ is positive only for tall matrices,
when $N \ge (C/c) n$. To reach out to matrices of arbitrary dimensions, one needs
to develop much more sensitive versions of the $\e$-net arguments. Nevertheless,
the end result stated in Theorem~\ref{t: subgaussian} exhibits the same two forces played
against one another -- the probability quantified by the dimension $N$ and the complexity of the sphere
$S^{n-1}$ quantified by its dimension $n-1$.

\subsection{Small ball probabilities, distance problems, and additive structure}

Our proof of Theorem~\ref{t: subgaussian} is a development of our method in \cite{RV}
for square matrices. Dealing with rectangular matrices is in several ways considerably
harder. Several new tools are developed in this paper, which may be of independent
interest.

\medskip

One new key ingredient is a {\em small ball probability} bound for
sums of independent random vectors in $\R^d$. We consider the sum $S
= \sum_k a_k X_k$ where $X_k$ are i.i.d. random variables and $a_k$
are real coefficients. We then estimate the probability that such
sum falls into a given small Euclidean ball in $\R^d$. Useful upper
bounds on the small ball probability must depend on the additive
structure of the coefficients $a_k$. The less structure the
coefficients carry, the more spread the distribution of $S$ is, so
the smaller is the small ball probability. Our treatment of small
ball probabilities is a development of the Littlewood-Offord theory
from \cite{RV}, which is now done in arbitrary dimension $d$ as
opposed in $d=1$ in \cite{RV}. While this paper was being written,
Friedland and Sodin \cite{FS} proposed two different ways to
simplify and improve our argument in \cite{RV}. With their kind permission,
we include in Section~\ref{s: sbp} a multi-dimensional version of an
unpublished argument of Friedland and Sodin \cite{FS private}, which is
considerably simpler than our original proof.

\medskip

We use small the ball probability estimates to prove an optimal
bound for the distance problem: {\em how close is a random vector
from an independent random subspace?} Consider a vector $X$ in
$\R^N$ with independent identically distributed coordinates and a
subspace $H$ spanned by $N-m$ independent copies of $X$. In
Section~\ref{s: distance}, we show that the distance is at least of
order $\sqrt{m}$ with high probability, and we obtain the sharp
estimate on this probability:
\begin{equation}                    \label{intro dist}
  \P \big( \dist(X,H) < \e \sqrt{m} \big)
  \le (C\e)^m + e^{-cN}.
\end{equation}
This bound is easy for a standard normal vector $X$ in $\R^N$, since $\dist(X,H)$
is in this case the Euclidean norm of the standard normal vector in $\R^m$.
However, for discrete distributions, such as for $X$ with $\pm 1$ random coordinates,
estimate \eqref{intro dist} is non-trivial.
In \cite{RV}, it was proved for $m=1$; in this paper we extend
the distance bound to all dimensions.

To prove \eqref{intro dist}, we first use the small ball probability inequalities
to compute the {\em distance to an arbitrary subspace $H$}. This estimate
necessarily depends on the additive structure of the subspace $H$; the less
structure, the better is our estimate, see Theorem~\ref{distance general}.
We then prove the intuitively plausible fact that
{\em random subspaces have no arithmetic structure},
see Theorem~\ref{LCD random}. This together leads to the desired distance
estimate \eqref{intro dist}.

\medskip

The distance bound is then used to prove our main result, Theorem~\ref{t: subgaussian}.
Let $X$ be some column of the random matrix $A$ and $H$ be the span of the
other columns. The simple rank argument shows that the smallest singular
value $s_n(A) = 0$ if and only if $X \in H$ for some column.
A simple quantitative version of this argument is that
a lower estimate on $s_n(A)$ yields a lower bound on $\dist(X,H)$.

In Section~\ref{s: via dist}, we show how to reverse this argument
for random matrices -- {\em deduce a lower bound on the smallest singular
value $s_n(A)$ from lower bound \eqref{intro dist} on the distance $\dist(X,H)$}.
Our reverse argument is harder than its version for square matrices
from \cite{RV}, where we had $m=1$. First, instead of one column $X$
we now have to consider all linear combinations of $d \sim m/2$ columns;
see Lemma~\ref{l: via dist}.
To obtain a distance bound that would be uniformly good for
all such linear combinations, one would normally use an $\e$-net argument.
However, the distance to the $(N-m)$-dimensional subspace $H$
is not sufficiently stable for this argument to be useful for small $m$
(for matrices close to square).
We therefore develop a decoupling argument in Section~\ref{s: uniform distance}
to bypass this difficulty.

Once this is done, the proof is quickly completed in Section~\ref{s: completion}.

\subsection*{Acknowledgement}
We are grateful to Shuheng Zhou, Nicole Tomczak-Jae\-ger\-mann, Radoslaw Adamczak,
and the anonymous referee
for pointing out several inaccuracies in our argument.
The second named author is grateful for his wife Lilia for
her love and patience during the years this paper was being written.

\section{Notation and preliminaries}

Throughout the paper, positive constants are denoted $C, C_1, C_2, c, c_1, c_2, \ldots$
Unless otherwise stated, these are absolute constants. In some of our arguments
they may depend (polynomially) on specified parameters, such as the subgaussian moment $B$.

The canonical inner product on $\R^n$ is denoted $\< \cdot , \cdot
\> $, and the Euclidean norm on $\R^n$ is denoted $\|\cdot\|_2$. The
Euclidean distance from a point $a$ to a subset $D$ in $\R^n$ is
denoted $\dist(a,D)$. The Euclidean ball of radius $R$ centered at a
point $a$ is denoted $B(a,R)$. The unit Euclidean sphere centered at
the origin is denoted $S^{n-1}$. If $E$ is a subspace of $\R^n$, its
unit Euclidean sphere is denoted $S(E) := S^{n-1} \cap E$.

The orthogonal projection in $\R^n$ onto a subspace $E$ is denoted $P_E$.
For a subset of coordinates $J \subseteq \{1,\ldots,n\}$, we sometimes
write $P_J$ for $P_{\R^J}$ where it causes no confusion.

\subsection{Nets}

Consider a subset $D$ of $\R^n$, and let $\e > 0$.
Recall that an $\e$-net of $D$ is a subset $\NN \subseteq D$ such that
for every $x \in D$ one has $\dist(x,\NN) \le \e$.

The following Lemma is a variant of the well known volumetric estimate.

\begin{proposition}[Nets]                     \label{nets}
  Let $S$ be a subset of $S^{n-1}$, and let $\e > 0$.
  Then there exists an $\e$-net of $S$ of
  cardinality at most
  $$
  2n \Big( 1 + \frac{2}{\e} \Big)^{n-1}.
  $$
\end{proposition}

The published variants of his lemma (e.g. \cite{MS}, Lemma 2.6) have exponent
$n$ rather than $n-1$. Since the latter exponent will be crucial for our purposes, we
include the proof of this lemma for the reader's convenience.

\begin{proof}
   Without loss of generality we can assume that $\e<2$, otherwise
   any single point forms a desired net.
   Let $\NN$ be an $\e$-separated subset of $S$ of
   maximal cardinality. By maximality, $\NN$ is an $\e$-net of $S$.
   Since $\NN$ is $\e$-separated, the balls $B(x,\e/2)$ with
   centers $x \in \NN$ are disjoint.
   All these balls have the same volume, and they are contained in the spherical
   shell $B(0,1+\e/2) \setminus B(0,1-\e/2)$. Therefore, comparing the volumes, we have
   \[
     |\mathcal{N}| \cdot \vol(B(0,\e/2))
     \le \vol \big( B(0,1+\e/2) \setminus B(0,1-\e/2) \big).
   \]
   Dividing both sides of this inequality by $\vol(B(0,1))$, we obtain
   $$
   |\mathcal{N}| \cdot (\e/2)^n
   \le (1+\e/2)^n-(1-\e/2)^n.
   $$
   Using the inequality $(1+x)^n-(1-x)^n \le 2nx (1+x)^{n-1}$ valid for $x \in (0,1)$,
   we conclude that $|\NN|$ is bounded as desired.
   This completes the proof.
\end{proof}

The following well known argument allows one to compute the norm
of a linear operator using nets. We have not found a published reference to this argument,
so we include it for the reader's convenience.

\begin{proposition}[Computing norm on nets]     \label{norm on nets}
  Let $\NN$ be a $\e$-net of $S^{n-1}$ and $\MM$ be a $\d$-net of $S^{m-1}$.
  Then for any linear operator  $A : \R^n \to \R^m$
  $$
  \|A\| \le \frac{1}{(1-\e)(1-\d)} \sup_{x \in \NN, \, y \in \MM} |\< Ax,y \> |.
  $$
\end{proposition}

\begin{proof}
Every $z \in S^{n-1}$ has the form $z = x + h$, where $x \in \NN$
and $\|h\|_2 \le \e$.
Since $\|A\| = \sup_{z \in S^{n-1}} \|Az\|_2$, the triangle inequality yields
$$
\|A\| \le \sup_{x \in \NN} \|Ax\|_2 + \max_{\|h\|_2 \le \e} \|Ah\|_2.
$$
The last term in the right hand side is bounded by $\e \|A\|$. Therefore we have
shown that
$$
(1-\e)\|A\| \le \sup_{x \in \NN} \|Ax\|_2.
$$
Fix $x \in \NN$. Repeating the above argument for
$\|Ax\|_2 = \sup_{y \in S^{m-1}} |\< Ax,y \> |$ yields the bound
$$
(1-\d) \|Ax\|_2 \le \sup_{y  \in \MM} |\< Ax,y \> |.
$$
The two previous estimates complete the proof.
\end{proof}

Using nets, one easily proves the well known basic bound $O(\sqrt{N})$
on the norm of a random subgaussian matrix:

\begin{proposition}[Norm]                       \label{norm}
  Let $A$ be an $N \times n$ random matrix, $N \ge n$, whose elements
  are independent copies of a subgaussian random variable. Then
  $$
  \P \big( \|A\| > t \sqrt{N} \big)
  \le e^{-c_0 t^2 N}
  \qquad \text{for } t \ge C_0,
  $$
  where $C_0, c_0 > 0$ depend only on the subgaussian moment $B$.
\end{proposition}

\begin{proof}
Let $\NN$ be a $(1/2)$-net of $S^{N-1}$ and $M$ be a $(1/2)$-net of $S^{n-1}$.
By Proposition~\ref{nets}, we can choose these nets such that
$$
|\NN| \le 2N \cdot 5^{N-1} \le 6^N, \quad |\MM| \le 2n \cdot 5^{n-1}
\le 6^n.
$$
For every $x \in \NN$ and $y \in \MM$, the random variable
$\< Ax,y\> $ is subgaussian (see Fact~2.1 in \cite{LPRT}), thus
$$
\P \big( |\< Ax,y \> | > t \sqrt{N} \big) \le C_1 e^{-c_1 t^2 N}
\qquad \text{for } t > 0,
$$
where $C_1, c_1 > 0$ depend only on the subgaussian moment $B$.
Using Lemma~\ref{norm on nets} and taking the union bound, we obtain
$$
\P \big( \|A\| > t \sqrt{N} \big)
\le 4 |\NN| |\MM| \max_{x \in N, \, y \in M}
  \P \big( |\< Ax,y \> |  > t \sqrt{d} \big)
\le 4 \cdot 6^N \cdot 6^N \cdot  C_1 e^{-c_1 t^2 N}.
$$
This completes the proof.
\end{proof}

\subsection{Compressible and incompressible vectors}

In our proof of Theorem~\ref{t: subgaussian}, we will make use
of a partition of the unit sphere $S^{n-1}$ into two sets of
compressible and incompressible vectors.
These sets were first defined in \cite{RV} as follows.

\begin{definition}[Compressible and incompressible vectors]
\label{d: compressible}
  Let $\d, \rho \in (0,1)$.
  A vector $x \in \R^n$ is called {\em sparse} if
  $|\supp(x)| \le \d n$.
  A vector $x \in S^{n-1}$ is called {\em compressible} if $x$
  is within Euclidean distance $\rho$ from the set of
  all sparse vectors.
  A vector $x \in S^{n-1}$ is called {\em incompressible}
  if it is not compressible.
  The sets of compressible and incompressible vectors
  will be denoted by $\Comp = \Comp(\d,\rho)$
  and $\Incomp = \Incomp(\d,\rho)$ respectively.
\end{definition}

We now recall without proof two simple results.
The first is Lemma~3.4 from \cite{RV}:

\begin{lemma}[Incompressible vectors are spread]    \label{l: spread}
  Let $x \in \Incomp(\d,\rho)$.
  Then there exists a set $\s = \s(x) \subseteq \{1, \ldots, n\}$
  of cardinality $|\s| \ge \frac{1}{2} \rho^2 \d n$ and such that
  \begin{equation}                                  \label{eq: spread}
    \frac{\rho}{\sqrt{2n}} \le |x_k| \le \frac{1}{\sqrt{\d n}}
    \qquad \text{for all $k \in \s$.}
  \end{equation}
\end{lemma}

The other result is a variant of Lemma~3.3 from \cite{RV}, which
establishes the invertibility on compressible vectors, and allows us
to focus on incompressible vectors in our proof of Theorem~\ref{t:
subgaussian}. While Lemma~3.3 was formulated in \cite{RV} for a
square  matrix, the same proof applies to $N \times n$ matrices,
provided that $ N \ge n/2$.

\begin{lemma}[Invertibility for compressible vectors]       \label{l: compressible}
    Let $A$ be an $N \times n$ random matrix, $ N \ge n/2$, whose elements
  are independent copies of a subgaussian random variable.
  There exist $\d, \rho, c_3 > 0$ depending only on the subgaussian moment $B$ such that
  $$
  \P \big( \inf_{x \in \Comp(\d,\rho)} \|A x\|_2 \le c_3 \sqrt{N} \big)
  \le e^{-c_3 N}.
  $$
\end{lemma}
\qed

\section{Small ball probability and the arithmetic structure}       \label{s: sbp}

Starting from the works of L\'evy \cite{Lev}, Kolmogorov \cite{Kol} and Ess\'{e}en \cite{Ess}, a number of
results in probability theory was concerned with the question how
spread the sums of independent random variables are. It is
convenient to quantify the spread of a random variable in the
following way.

\begin{definition}
  The {\em L\'evy concentration function} of a random vector $S$ in $\R^m$
  is defined for $\e > 0$ as
  $$
  \LL(S, \e) = \sup_{v \in \R^m} \P ( \|S-v\|_2 \le \e ).
  $$
\end{definition}

An equivalent way of looking at the L\'evy concentration function
is that it measures the {\em small ball probabilities} -- the likelihood
that the random vector $S$ enters a small ball in the space.
An exposition of the theory of small ball probabilities can
be found in \cite{LS}.

One can derive a simple but rather weak bound on L\'evy concentration function
from Paley-Zygmund inequality.

\begin{lemma}                               \label{nontrivial conc}
  Let $\xi$ be a random variable with mean zero, unit variance, and
  finite fourth moment. Then for every $\e \in (0,1)$ there exists
  $p \in (0,1)$ which depends only on $\e$ and on the fourth moment,
  and such that
  $$
  \LL(\xi, \e) \le p.
  $$
\end{lemma}

\begin{remark}
  In particular, this bound holds for subgaussian random variables,
  and with $p$ that depends only on $\e$ and the subgaussian moment.
\end{remark}

\begin{proof}
We use Paley-Zygmund inequality, which states for a random variable $Z$ that
\begin{equation}                                    \label{paley-zygmund}
  \P (|Z| > \e)
  \ge \frac{(\E Z^2 - \e^2)^2}{\E Z^4}, \qquad \e > 0,
\end{equation}
see e.g. \cite{LPRT}, Lemma 3.5.

Let $v \in \R$ and consider the random variable $Z = \xi - v$. Then
$$
\E Z^2 = 1 + v^2.
$$
By H\"older inequality, we have
$$
B := \E \xi^4 \ge (\E \xi^2)^2 = 1,
$$
so, using Minkowski inequality, we obtain
$$
(\E Z^4)^{1/4} \le B^{1/4} + v
\le B^{1/4} (1+v)
\le B^{1/4} 2^{1/2} (1+v^2)^{1/2}.
$$
Using this in \eqref{paley-zygmund}, we conclude that
$$
\P(|\xi-v| > \e)
\ge \frac{(1 + v^2 - \e^2)^2}{4B(1+v^2)^2}
= \frac{1}{4B} \Big( 1 - \frac{\e^2}{1+v^2} \Big)^2
\ge \frac{1-\e^2}{4B}.
$$
This completes the proof.
\end{proof}

We will need a much stronger bound on the concentration function for
sums of independent random variables. Here we present a
multi-dimensional version of the inverse Littlewood-Offord
inequality from \cite{RV}. While this paper was in preparation,
Friedland and Sodin \cite{FS} proposed two different ways to
simplify and improve our argument in \cite{RV}. We shall therefore
present here a multi-dimensional version of one of arguments of
Friedland and Sodin \cite{FS private}, which is considerably simpler
than our original proof.

We consider the sum
$$
S = \sum_{k=1}^N a_k \xi_k
$$
where $\xi_k$ are independent and identically distributed random variables,
and $a_k$ are some vectors in $\R^m$.
The Littlewood-Offord theory describes the behavior of the L\'evy concentration
function of $S$ in terms of the additive structure of the vectors $a_k$.

In the scalar case, when $m=1$,
the additive structure of a sequence $a = (a_1,\ldots,a_N)$ of real numbers $a_k$
can be described in terms of the shortest arithmetic progression into which
it (essentially) embeds. This length is conveniently expressed as the
essential {\em least common denominator} of $a$, defined as follows.
We fix parameters $\a, \g \in (0,1)$, and define
$$
\LCD_{\a,\g}(a)
:= \inf \Big\{ \theta > 0: \; \dist (\theta a, \Z^N) < \min(\g\|\theta a\|_2,\a) \Big\}.
$$
The requirement that the distance is smaller than $\g\|\theta a\|_2$
forces to consider only non-trivial integer points as
approximations of $\theta a$ -- only those in a non-trivial cone
around the direction of $a$. One typically uses this definition with
$\g$ a small constant, and for $\a = c \sqrt{N}$ with a small constant $c>0$.
The inequality $\dist(\theta a, \Z^N) < \a$ then yields that most coordinates of
$\theta a$ are within a small constant distance from integers.

The definition of the essential least common denominator
carries over naturally to higher dimensions and thus
allows one to control the arithmetic structure
of a sequence $a = (a_1,\ldots,a_N)$ of vectors $a_k \in \R^m$.
To this end, we define the product of such multi-vector $a$ and a vector $\theta \in \R^m$ as
$$
\theta \cdot a = (\< \theta, a_1\> , \ldots, \< \theta, a_N\> ) \in \R^N.
$$
A more traditional way of looking at $\theta \cdot a$ is to regard
it as the product of the matrix $a$ with rows $a_k$ and the vector
$\theta$.

Then we define, for $\a > 0$ and $\g \in (0,1)$,
$$
\LCD_{\a,\g}(a)
:= \inf \Big\{ \|\theta\|_2: \; \theta \in \R^m,
     \dist(\theta \cdot a, \Z^N) < \min(\g\|\theta \cdot a\|_2,\a) \Big\}.
$$

The following theorem gives a bound on the small ball probability
for a random sum $S = \sum_{k=1}^N a_k \xi_k$ in terms of the additive
structure of the coefficient sequence $a$.
The less structure in $a$, the bigger its least common denominator is,
and the smaller is the small ball probability for $S$.

\begin{theorem}[Small ball probability]                         \label{SBP}
  Consider a sequence $a = (a_1,\ldots,a_N)$ of vectors $a_k \in \R^m$,
  which satisfies
  \begin{equation}                                  \label{super-isotropy}
    \sum_{k=1}^N \pr{a_k}{x}^2 \ge \norm{x}_2^2
    \qquad \text{for every $x \in \R^m$.}
  \end{equation}
  Let $\xi_1, \ldots, \xi_N$ be independent and identically distributed,
  mean zero random variables, such that
  $\LL(\xi_k,1) \le 1-b$ for some $b > 0$.
  Consider the random sum $S = \sum_{k=1}^N a_k \xi_k$.
  Then, for every $\a > 0$ and $\g \in (0,1)$, and for
  $$
  \e \ge \frac{\sqrt{m}}{\LCD_{\a,\g}(a)},
  $$
  we have
  $$
  \LL(S,\e \sqrt{m})
  \le \Big( \frac{C\e}{\g \sqrt{b}} \Big)^m + C^m e^{-2b\a^2}.
  $$
\end{theorem}

\begin{remark}
  The non-degeneracy condition \eqref{super-isotropy} is meant to guarantee that the system
  of vectors $(a_k)$ is genuinely $m$-dimensional. It disallows these vectors
  to lie on or close to any lower-dimensional subspace of $\R^m$.

  Hal\'{a}sz \cite{H} developed a powerful approach to bounding
  concentration function; his approach influenced our arguments below.
  Hal\'{a}sz \cite{H} operated under a similar non-degeneracy condition on the vectors $a_k$:
  for every $x \in S^{m-1}$, at least $cN$ terms satisfy $|\pr{a_k}{x}| \ge 1$.
  After properly rescaling $a_k$ by the factor $\sqrt{c/N}$,
  Hal\'{a}sz's condition is seen to be more restrictive than
  \eqref{super-isotropy}.
\end{remark}

\subsection{Proof of the Small Ball Probability Theorem}

To estimate the L\'{e}vy concentration function we apply the
Ess\'{e}en Lemma, see e.g. \cite{TV}, p.~290.
\begin{lemma}  \label{l: Esseen}
  Let $Y$ be a random vector in $\R^m$. Then
  \[
    \sup_{v \in \R^m} \P (\norm{Y-v}_2 \le \sqrt{m})
    \le C^m \int_{B(0, \sqrt{m})} |\phi_Y(\theta)| \, d \theta
  \]
  where $\phi_Y(\theta) = \E \exp (2\pi i \pr{\theta}{Y} )$ is the
  characteristic function of $Y$.
\end{lemma}

Applying Lemma \ref{l: Esseen} to the vector $Y=S/\e$ and using the
independence of random variables $\xi_1 \etc \xi_N$, we obtain
\begin{equation}  \label{esseen}
  \LL(S,\e \sqrt{m})
  \le C^m \int_{B(0, \sqrt{m})} \prod_{k=1}^N
  |\phi(\pr{\theta}{a_k}/\e)| \, d \theta,
\end{equation}
where $\phi(t)=\E \exp (2 \pi i t \xi)$ is the characteristic
function of $\xi := \xi_1$. To estimate this characteristic function,
we follow the conditioning argument of \cite{R}, \cite{RV}. Let $\xi'$ be an independent copy
of $\xi$ and denote by $\bar{\xi}$ the symmetric random variable $\xi-\xi'$.
Then
\[
  |\phi(t)|^2=\E \exp (2 \pi i t \bar{\xi})
  = \E \cos (2 \pi  t \bar{\xi}).
\]
Using the inequality $|x| \le \exp (- \frac{1}{2} (1-x^2))$, which is valid for all
$x \in \R$, we obtain
\[
    |\phi(t)|
    \le \exp \Big( -\frac{1}{2} \big( 1- \E \cos (2 \pi  t \bar{\xi}) \big) \Big).
\]
By conditioning on $\xi'$ we see that our assumption $\LL(\xi,1)
\le 1-b$ implies that $\P(|\bar{\xi}| \ge 1) \ge b$. Therefore
\begin{align*}
  1- \E \cos (2 \pi  t \bar{\xi})
  &\ge \P(|\bar{\xi}| \ge 1) \cdot
    \E \Big (1- \cos (2 \pi  t \bar{\xi}) \mid |\bar{\xi}| \ge 1 \Big )
  \\
  &\ge b \cdot \frac{4}{\pi^2}
    \E \Big (\min_{q \in \Z} | 2 \pi  t \bar{\xi} - 2 \pi q|^2 \mid |\bar{\xi}| \ge 1 \Big )
  \\
  &= 16 b \cdot
    \E \Big (\min_{q \in \Z} |  t \bar{\xi} -  q|^2
        \mid |\bar{\xi}| \ge 1 \Big ).
\end{align*}
Substituting of this into \eqref{esseen} and using Jensen's
inequality, we get
\begin{align*}
  &\LL(S,\e \sqrt{m}) \\
  &\le C^m \int_{B(0, \sqrt{m})}
    \exp \Big( -8b \E \Big (
      \sum_{k=1}^N \min_{q \in \Z} |\bar{\xi} \pr{\theta}{a_k}/\e  -  q|^2
      \; \Big| \; |\bar{\xi}| \ge 1 \Big ) \Big) \, d \theta
  \\
  &\le C^m \E \Big( \int_{B(0, \sqrt{m})}
    \exp \Big( -8b \min_{p \in \Z^N} \Big\| \frac{\bar{\xi}}{\e} \, \theta \cdot a
- p \Big\|_2 \Big)
       \, d \theta
       \; \Big | \; |\bar{\xi}| \ge 1 \Big)
  \\
  &\le C^m \sup_{z \ge 1} \int_{B(0,\sqrt{m})} \exp(- 8 b f^2(\theta)) \; d\theta,
\end{align*}
where
$$
f(\theta) = \min_{p \in \Z^N} \Big\| \frac{z}{\e} \, \theta \cdot a
- p \Big\|_2.
$$

The next and major step is to bound the size of the {\em recurrence set}
$$
I(t) := \Big\{ \theta \in B(0,\sqrt{m}) : \; f(\theta) \le t \}.
$$

\begin{lemma}[Size of the recurrence set]               \label{size recurrence}
  We have
  $$
  \vol(I(t)) \le \Big( \frac{Ct\e}{\g \sqrt{m}} \Big)^m, \qquad t < \a/2.
  $$
\end{lemma}

\begin{proof}
Fix $t < \a/2$. Consider two points $\theta', \theta'' \in I(t)$.
There exist $p', p'' \in \Z^N$ such that
$$
\Big\| \frac{z}{\e} \, \theta' \cdot a - p' \Big\|_2 \le t, \quad
\Big\| \frac{z}{\e} \, \theta'' \cdot a - p'' \Big\|_2 \le t.
$$
Let
$$
\t := \frac{z}{\e} \, (\theta' - \theta''), \qquad p := p' - p''.
$$
Then, by the triangle inequality,
\begin{equation}                            \label{<2t}
  \|\t \cdot a - p\|_2 \le 2t.
\end{equation}
Recall that by the assumption of the theorem,
$$
\LCD_{\a,\g}(a) \ge \frac{\sqrt{m}}{\e}.
$$
Therefore, by the definition of the least common denominator, we have that
either
$$
\|\t\|_2 \ge \frac{\sqrt{m}}{\e},
$$
or otherwise
\begin{equation}                            \label{>min}
  \|\t \cdot a - p\|_2 \ge \min(\g\|\t \cdot a\|_2,\a).
\end{equation}
In the latter case, since $2t <\a$, inequalities \eqref{<2t} and \eqref{>min}
together yield
$$
2t \ge \g\|\t \cdot a\|_2 \ge \g \|\tau\|_2,
$$
where the last inequality follows from condition
\eqref{super-isotropy}.

Recalling the definition of $\t$, we have proved that every pair of
points $\theta', \theta'' \in I(t)$ satisfies:
$$
\text{either} \quad \|\theta' - \theta''\|_2 \ge \frac{\sqrt{m}}{z} =: R
\quad \text{or} \quad \|\theta' - \theta''\|_2 \le \frac{2t\e}{\g z} =: r.
$$
It follows that $I(t)$ can be covered by Euclidean balls of radii
$r$,  whose centers are $R$-separated in the Euclidean distance.
Since $I(t) \subset B(0,\sqrt{m})$, the number of such balls is at
most
\[
  \frac{\vol (B(0, \sqrt{m} + R/2))}{\vol (B(0,  R/2))}
  =\Big( \frac{2\sqrt{m}}{R} + 1 \Big)^m
  \le \Big( \frac{3\sqrt{m}}{R} \Big)^m.
\]
(In the last inequality we used that $R \le \sqrt{m}$ because $z \ge 1$).
Recall that the volume of a Euclidean ball of radius $r$ in $\R^m$ is bounded by
$(Cr/\sqrt{m})^m$.
Summing these volumes, we conclude that
$$
\vol(I(t)) \le \Big( \frac{3Cr}{R} \Big)^m,
$$
which completes the proof of the lemma.
\end{proof}

\begin{proof}[Proof of Theorem~\ref{SBP}.]
We decompose the domain into two parts. First, by the definition
of $I(t)$, we have
\begin{align}                           \label{outside recurrence}
\int_{B(0,\sqrt{m}) \setminus I(\alpha/2)} \exp(- 8b f^2(\theta)) \; d\theta
  &\le \int_{B(0,\sqrt{m})} \exp(-2b\a^2) \; d\theta \nonumber \\
  &\le C^m \exp(-2b\a^2).
\end{align}
In the last line, we used the estimate $|\vol(B(0,\sqrt{m})| \le C^m$.

Second, by the integral distribution formula and using Lemma~\ref{size recurrence},
we have
\begin{align}                           \label{on recurrence}
\int_{I(\alpha/2)} \exp(- 8 b f^2(\theta)) \; d\theta
  &= \int_0^{\a/2} 16 b t \exp(-8 b t^2) |\vol(I(t))| \; dt \nonumber \\
  &\le 16 b \Big( \frac{C\e}{\g \sqrt{m}} \Big)^m
    \int_0^\infty t^{m+1} \exp(-8 b t^2) \; dt \nonumber \\
  &\le \Big( \frac{C'\e}{\g \sqrt{b}} \Big)^m \sqrt{m}
  \le \Big( \frac{C''\e}{\g \sqrt{b}} \Big)^m.
\end{align}
Combining \eqref{outside recurrence} and \eqref{on recurrence}
completes the proof of Theorem~\ref{SBP}.
\end{proof}

\subsection{Least common denominator of incompressible vectors}

We now prove a simple fact that the least common denominator
of any incompressible vector $a$ in $\R^N$ is at least of order $\sqrt{N}$.
Indeed, by Lemma~\ref{l: spread} such a vector has many
coordinates of order $1/\sqrt{N}$. Therefore,
to make a dilation $\theta a$ of this vector close to an integer point,
one has to scale $a$ by at least $\theta  \gtrsim \sqrt{N}$.
We now make this heuristic reasoning formal.

\begin{lemma}[LCD of incompressible vectors]                    \label{LCD sqrt N}
  For every $\d,\rho \in (0,1)$ there exist $c_1(\d,\rho) > 0$ and $c_2(\d) > 0$
  such that the following holds.
  Let $a \in \R^N$ be an incompressible vector: $a \in \Incomp(\d,\rho)$.
  Then, for every  $0 < \g < c_1(\d,\rho)$ and every $\a > 0$, one has
  $$
  \LCD_{\a,\g}(a) > c_2(\d) \sqrt{N}.
  $$
\end{lemma}

\begin{remark}
  The proof gives $c_1(\d,\rho) = \frac{1}{2} \rho^2 \sqrt{\d}$
  and $c_2(\d) = \frac{1}{2} \sqrt{\d}$.
\end{remark}

\begin{proof}
By Lemma~\ref{l: spread}, there exists a set $\s_1 \subseteq
\{1,\ldots, N\}$ of size
$$
|\s_1| \ge \frac{1}{2} \rho^2 \d N
$$
and such that
\begin{equation}                            \label{spread}
  \frac{\rho}{\sqrt{2N}} \le |a_k| \le \frac{1}{\sqrt{\d N}}
  \qquad \text{for $k \in \s_1$.}
\end{equation}

Let $\theta := \LCD_{\a,\g}(a)$. Then there exists $p \in \Z^N$ such that
$$
\|\theta a - p\|_2 < \g \|\theta a\|_2 = \g \theta.
$$
This shows in particular that $\theta > 0$; dividing by $\theta$ gives
$$
\Big\| a - \frac{p}{\theta} \Big\|_2 < \g.
$$
Then by Chebychev inequality, there exists a set $\s_2 \subseteq \{1,\ldots, N\}$
of size
$$
|\s_2| > N - \frac{1}{2} \rho^2 \d N
$$
and such that
\begin{equation}                            \label{diophantine}
\Big| a_k - \frac{p_k}{\theta} \Big|
  < \frac{\sqrt{2}}{\rho \sqrt{\d}} \cdot \frac{\g}{\sqrt{N}}
  \qquad \text{for $k \in \s_2$.}
\end{equation}

Since $|\s_1| + |\s_2| > N$, there exists $k \in \s_1 \cap \s_2$.
Fix this $k$.
By the left hand side of \eqref{spread}, by \eqref{diophantine} and the
assumption on $\g$ we have:
$$
\Big| \frac{p_k}{\theta} \Big| \ge \frac{\rho}{\sqrt{2N}} -
\frac{\sqrt{2}}{\rho \sqrt{\d}} \cdot \frac{\g}{\sqrt{N}}
> 0.
$$
Thus $|p_k| > 0$; since $p_k$ is an integer, this yields $|p_k| \ge 1$.
Similarly, using the right hand side of \eqref{spread}, \eqref{diophantine} and the
assumption on $\g$, we get
$$
\Big| \frac{p_k}{\theta} \Big| \le \frac{1}{\sqrt{\d N}} +
\frac{\sqrt{2}}{\rho \sqrt{\d}} \cdot \frac{\g}{\sqrt{N}} <
\frac{2}{\sqrt{\d N}}.
$$
Since $|p_k| \ge 1$, this yields
$$
|\theta| > \frac{\sqrt{\d N}}{2}.
$$
This completes the proof.
\end{proof}

\section{The distance problem and arithmetic structure}  \label{s: distance}

Here we use the Small Ball Probability Theorem \ref{SBP} to give an optimal
bound for the distance problem: {\em how close is a random vector $X$ in $\R^N$
from an independent random subspace $H$ of codimension $m$}?

If $X$ has the standard normal distribution, then the distance
does not depend on the distribution of $H$. Indeed, for an arbitrary fixed $H$,
the distance $\dist(X,H)$ is distributed identically with the Euclidean norm of a standard
normal random vector in $\R^m$. Therefore,
$$
\dist(X,H) \sim \sqrt{m} \qquad \text{with high probability}.
$$
More precisely, standard computations give for every $\e > 0$ that
\begin{equation}                                    \label{dist normal}
  \P \big( \dist(X,H) < \e \sqrt{m} \big) \le (C\e)^m.
\end{equation}
However, if $X$ has a more general distribution with independent coordinates,
the distance $\dist(X,H)$ may strongly depend on the subspace $H$.
For example, if the coordinates of $X$ are $\pm 1$ symmetric random variables.
then for $H = \{ x:\; x_1 + x_2 = 0\}$ the distance equals $0$ with probability $1/2$,
while for $H = \{ x:\; x_1 + \cdots + x_N = 0\}$ the distance equals $0$ with probability
$\sim 1/\sqrt{N}$.

Nevertheless, a version of the distance bound \eqref{dist normal}
remains true for general distributions if $H$ is a random subspace. For spaces
of codimension $m=1$, this result was proved in \cite{RV}. In this paper,
we prove an optimal distance bound for general dimensions.

\begin{theorem}[Distance to a random subspace]                       \label{distance}
  Let $X$ be a vector in $\R^N$ whose coordinates
  are independent and identically distributed mean zero subgaussian random variables
  with unit variance.
  Let $H$ be a random subspace in $\R^N$ spanned by $N-m$ vectors, $0< m < \tilde{c} N$,
  whose coordinates are independent and identically distributed mean zero subgaussian
  random variables with unit variance,
  independent of $X$.
  Then, for every $v \in \R^N$ and every $\e > 0$, we have
  $$
  \P \big( \dist(X,H+v) < \e \sqrt{m} \big)
  \le (C\e)^m + e^{-cN},
  $$
  where $C, c, \tilde{c} > 0$ depend only on the subgaussian moments.
\end{theorem}

\begin{remark}
  To explain the term $e^{-cN}$, consider $\pm 1$ symmetric random variables.
  Then with probability at least $2^{-n}$ the random vector $X$ coincides with one of
  the random vectors that span $H$, which makes the distance equal zero.
\end{remark}

We will deduce Theorem~\ref{distance} from a more general inequality that holds for
{\em arbitrary} fixed subspace $H$. This bound will depend on the arithmetic
structure of the subspace $H$, which we express using the least common
denominator.

For $\a > 0$ and $\g \in (0,1)$, the essential {\em least common denominator of a subspace}
$E$ in $\R^N$ is defined as
$$
\LCD_{\a,\g} (E) := \inf \{ \LCD_{\a,\g} (a) :\; a \in S(E) \}.
$$
Clearly,
$$
\LCD_{\a,\g}(E)
= \inf \Big\{ \|\theta\|_2 :\;
    \theta \in E, \, \dist(\theta, \Z^N) < \min(\g\|\theta\|_2,\a) \Big\}.
$$
Then Theorem~\ref{SBP} quickly leads to the following general distance bound:

\begin{theorem}[Distance to a general subspace]         \label{distance general}
  Let $X$ be a vector in $\R^N$ whose coordinates
  are independent and identically distributed mean zero subgaussian random variables
  with unit variance.
  Let $H$ be a subspace in $\R^N$ of dimension $N-m > 0$.
  Then for every  $v \in \R^N$, $\a > 0$, $\g \in (0,1)$, and for
  $$
  \e \ge \frac{\sqrt{m}}{\LCD_{\a,\g}(H^\perp)},
  $$
  we have
  $$
  \P \big( \dist(X,H+v) < \e \sqrt{m} \big)
  \le \Big( \frac{C\e}{\g} \Big)^m + C^m e^{-c\a^2}
  $$
  where $C, c > 0$ depend only on the subgaussian moment.
\end{theorem}

\begin{proof}
Let us write $X$ in coordinates, $X = (\xi_1,\ldots,\xi_N)$.
By Lemma~\ref{nontrivial conc} and the remark below it,
all coordinates of $X$ satisfy the inequality
$\LL(\xi_k,1/2) \le 1 - b$ for some $b > 0$ that depends only on
the subgaussian moment of $\xi_k$.
Hence the random variables $\xi_k/2$ satisfy the assumption in Theorem~\ref{SBP}.

Next, we connect the distance to a sum of independent random vectors:
\begin{equation}                    \label{dist as sum}
  \dist(X,H+v) = \|P_{H^\perp} (X-v)\|_2 = \Big\| \sum_{k=1}^N a_k \xi_k - w \Big\|_2,
\end{equation}
where
$$
a_k = P_{H^\perp} e_k, \quad w = P_{H^\perp} v,
$$
and where $e_1,\ldots,e_N$ denotes the canonical basis of $\R^N$. Therefore,
the sequence of vectors $a = (a_1, \ldots, a_N)$ is in the isotropic
position:
$$
\sum_{k=1}^N \pr{a_k}{x}^2 = \norm{x}_2^2 \qquad \text{for any }
  x \in H^\perp,
$$
so we can use Theorem~\ref{SBP} in the space $H^\perp$ (identified with $\R^m$
by a suitable isometry).

For every $\theta = (\theta_1,\ldots,\theta_N) \in H^\perp$ and every $k$ we have
$
\< \theta, a_k \> = \< P_{H^\perp} \theta, e_k \>
= \< \theta, e_k \>  = \theta_k,
$
so
$$
\theta \cdot a = \theta
$$
where the right hand side is considered as a vector in $\R^N$.
Therefore the least common denominator of a subspace can be
expressed by that of a sequence of vectors $a=(a_1 \etc a_N)$:
$$
\LCD_{\a,\g}(H^\perp) = \LCD_{\a,\g}(a).
$$
The theorem now follows directly from Theorem~\ref{SBP}.
\end{proof}

In order to deduce the Distance Theorem~\ref{distance}, it will now suffice to
bound below the least common denominator of a random subspace $H^\perp$.
Heuristically, the randomness should remove any arithmetic
structure from the subspace,
thus making the least common denominator exponentially large. Our next results
shows that this is indeed true.

\begin{theorem}[Structure of a random subspace]              \label{LCD random}
  Let $H$ be a random subspace in $\R^N$ spanned by $N-m$ vectors, $1 \le m < \tilde{c} N$,
  whose coordinates are independent and identically distributed mean zero subgaussian
  random variables with unit variance.
  Then, for $\a = c \sqrt{N}$, we have
  $$
  \P \big( \LCD_{\a,c}(H^\perp) < c \sqrt{N} e^{cN/m} \big)
  \le e^{-cN},
  $$
  where $c \in (0,1)$ and $\tilde{c} \in (0,1/2)$ depend only on the subgaussian moment.
\end{theorem}

Assuming that this result holds, we can complete the proof of the
Distance Theorem~\ref{distance}.

\begin{proof}[Proof of Theorem~\ref{distance}.]
Consider the event
$$
\EE := \big\{ \LCD_{\a,c}(H^\perp) \ge c \sqrt{N} e^{cN/m} \big\}.
$$
By Theorem~\ref{LCD random}, $\P(\EE^c) \le e^{-cN}$.

Let us condition on a realization of $H$ in $\EE$. By the independence of
$X$ and $H$, Theorem~\ref{distance general} used with $\a = c \sqrt{N}$
and $\g = c$ gives
$$
\P \big( \dist(X,H) < \e \sqrt{m} \;|\; \EE \big) \le (C_1 \e)^m +
C^m e^{-c_1 N}
$$
for every
$$
\e > C_2 \sqrt{\frac{m}{N}} e^{-cN/m}.
$$
Since $m \le \tilde{c} N$, with an appropriate choice of $\tilde{c}$
we get
\[
  C_2 \sqrt{\frac{m}{N}} e^{-cN/m} \le \frac{1}{C_1} e^{-c_3 N/m}
  \quad \text{and }
  C^m e^{-c_1 N} \le e^{-c_3 N}.
\]
 Therefore, for every $\e
> 0$,
$$
\P \big( \dist(X,H) < \e \sqrt{m} \;|\; \EE \big) \le (C_1 \e)^m + 2
e^{-c_3 N}
  \le (C_1 \e)^m +
e^{-c_4 N}.
$$
By the estimate on the probability of $\EE^c$, this completes the proof.
\end{proof}

\subsection{Proof of the Structure Theorem \ref{LCD random}}

Note first, that throughout the proof we can assume that $N>N_0$,
where $N_0$ is a suitably large number, which may depend only the
subgaussian moment. Indeed, the assumption on $m$ implies that
$N > m/\tilde{c} \ge 1/\tilde{c}$. Choosing $\tilde{c}>0$ suitably small
depending on the subgaussian moment,  we can make $N_0$ suitably
large.

Let $X_1,\ldots,X_{N-m}$ denote the independent random vectors that
span the subspace $H$. Consider an $(N-m) \times N$ random matrix
$B$ with rows $X_k$. Then
$$
H^\perp \subseteq \ker(B).
$$
Therefore, for every set $S$ in $\R^N$ we have:
\begin{equation}                                        \label{navigation}
 \inf_{x \in S} \|Bx\|_2 > 0
  \text{ implies } H^\perp \cap S = \emptyset.
\end{equation}
This observation will help us to ``navigate'' the random subspace
$H^\perp$ away from undesired sets $S$ on the unit sphere.

We start with a variant of Lemma~3.6 of \cite{RV}; here we use the
concept of compressible and incompressible vectors in $\R^N$ rather
than $\R^n$.

\begin{lemma}[Random subspaces are incompressible]      \label{random incompressible}
  There exist $\d, \rho \in (0,1)$ such that
  $$
  \P \big( H^\perp \cap S^{N-1} \subseteq \Incomp(\d,\rho) \big)
  \ge 1 - e^{-cN}.
  $$
\end{lemma}

\begin{proof}
  Let $B$ be the $(N-m) \times N$ matrix defined above. Since $N-m
  > (1- \tilde{c}) N$ and $\tilde{c}<1/2$, we can apply Lemma \ref{l: compressible} for
  the matrix $B$. Thus, there exist $\d, \rho \in (0,1)$ such that
  $$
  \P \big( \inf_{x \in \Comp (\d, \rho)} \norm{Bx}_2 \ge c_3 \sqrt{N} \big)
  \ge 1- e^{-c_3 N}.
  $$
  By \eqref{navigation}, $H^\perp \cap \Comp (\d, \rho) = \emptyset$
  with probability at least $1- e^{-c_3 N}$.
\end{proof}

Fix the values of $\d$ and $\rho$ given by Lemma~\ref{random incompressible}
for the rest of this section. We will further decompose the set of incompressible vectors
into level sets $S_D$ according to the value of the least common denominator $D$.
We shall prove a nontrivial lower bound on $\inf_{x \in S_D} \|Bx\|_2 > 0$
for each level set up to $D$ of the exponential order. By \eqref{navigation},
this will mean that $H^\perp$ is disjoint from every such level set.
Therefore, all vectors in $H^\perp$ must have exponentially large least
common denominators $D$. This is Theorem~\ref{LCD random}.

\medskip

Let $\a = \mu \sqrt{N}$, where $\mu>0$ is a small number to be chosen
later, which depends only on the subgaussian moment. By Lemma~\ref{LCD sqrt N},
$$
\LCD_{\a,c}(x) \ge c_0 \sqrt{N} \qquad \text{ for every $x \in
\Incomp$}.
$$

\begin{definition}[Level sets]                  \label{def level sets}
  Let $D \ge c_0 \sqrt{N}$.
  Define $S_D \subseteq S^{N-1}$ as
  $$
  S_D := \big\{ x \in \Incomp :\; D \le \LCD_{\a,c}(x) < 2D \big\}.
  $$
\end{definition}

To obtain a lower bound for $\|Bx\|_2$ on the level set, we proceed by an $\e$-net argument.
To this end, we first need such a bound for a single vector $x$.

\begin{lemma}[Lower bound for a single vector]          \label{for single vector}
  Let $x \in S_D$. Then for every $t > 0$ we have
  \begin{equation}                          \label{eq for single vector}
    \P \big( \|Bx\|_2 < t \sqrt{N} \big)
    \le \Big( Ct + \frac{C}{D} + Ce^{-c\a^2} \Big)^{N-m}.
  \end{equation}
\end{lemma}

\begin{proof}
Denoting the elements of $B$ by $\xi_{jk}$, we can write the $j$-th coordinate of $Bx$ as
$$
(Bx)_j = \sum_{j=1}^N \xi_{jk} x_k =: \zeta_j, \qquad j = 1,\ldots, N-m.
$$
Now we can use the Small Ball Probability Theorem~\ref{SBP} in dimension $m=1$
for each of these random sums.
By Lemma~\ref{nontrivial conc} and the remark below it,
$\LL(\xi_{jk},1/2) \le 1 - b$ for some $b > 0$ that depends only on
the subgaussian moment of $\xi_{jk}$.
Hence the random variables $\xi_{jk}/2$ satisfy the assumption in Theorem~\ref{SBP}.
This gives for every $j$ and every $t>0$:
$$
\P (|\zeta_j| < t) \le Ct + \frac{C}{\LCD_{\a,c}(x)} + C e^{-c\a^2}
\le Ct + \frac{C}{D} + C e^{-c\a^2}.
$$
Since $\zeta_j$ are independent random variables, we can use
Tensorization Lemma~2.2 of \cite{RV} to conclude that
for every $t>0$,
$$
\P \Big( \sum_{j=1}^{N-m} |\zeta_j|^2 < t^2 (N-m) \Big)
\le \Big( C''t + \frac{C''}{D} + C''e^{-c\a^2} \Big)^{N-m}.
$$
This completes the proof, because $\|Bx\|_2^2 = \sum_{j=1}^{N-m} |\zeta_j|^2$
and $N \le 2(N-m)$ by the assumption.
\end{proof}

Next, we construct a small $\e$-net of the level set $S_D$.
Since this set lies in $S^{N-1}$, Lemma~\ref{nets} yields the existence
of an $(\sqrt{N}/D)$-net of cardinality at most $(CD / \sqrt{N})^N$.
This simple volumetric bound is not sufficient for our purposes,
and this is the crucial step where we explore the additive structure of $S_D$
to construct a smaller net.

\begin{lemma}[Nets of level sets]                   \label{net}
  There exists a $(4\a/D)$-net of $S_D$ of cardinality at most $(C_0 D / \sqrt{N})^N$.
\end{lemma}

\begin{remark}
  Recall that $\a$ is chosen as a small proportion of $\sqrt{N}$. Hence
  Lemma~\ref{net} gives a better bound than the standard volumetric bound in
  Lemma~\ref{nets}.
\end{remark}

\begin{proof}
We can assume that $4\a/D \le 1$, otherwise the conclusion is trivial.
For $x \in S_D$, denote
$$
D(x) := \LCD_{\a,c}(x).
$$
By the definition of $S_D$, we have $D \le D(x) < 2D$.
By the definition of the least common denominator, there exists $p \in \Z^N$ such that
\begin{equation}                    \label{Dx x}
  \|D(x)x - p\|_2 < \a.
\end{equation}
Therefore
$$
\Big\| x - \frac{p}{D(x)} \Big\|_2
< \frac{\a}{D(x)} \le \frac{\a}{D} \le \frac{1}{4}.
$$
Since $\|x\|_2 = 1$, it follows that
\begin{equation}                    \label{x-pp}
  \Big\| x - \frac{p}{\|p\|_2} \Big\|_2
  < \frac{2\a}{D}.
\end{equation}

On the other hand, by \eqref{Dx x} and using that $\|x\|_2 = 1$,
$D(x) \le 2D$ and $4\a/D \le 1$, we obtain
\begin{equation}                    \label{net bounded}
  \|p\|_2 < D(x) + \a
  \le 2D + \a
  \le 3D.
\end{equation}

Inequalities \eqref{x-pp} and \eqref{net bounded} show that every point $x \in S_D$
is within Euclidean distance $2\a/D$ from the set
$$
\NN := \Big\{ \frac{p}{\|p\|_2} :\; p \in \Z^N \cap B(0,3D) \Big\}.
$$
A known volumetric argument gives a bound on the number of integer points in
$B(0,3D)$:
$$
|\NN| \le (1+9D/\sqrt{N})^N \le (C_0 D/\sqrt{N})^N
$$
(where in the last inequality we used that by Definition~\ref{def level sets} of the level sets, $D > c_0\sqrt{N}$).
Finally, there exists a $(4\a/D)$-net of $S_D$ with the same cardinality as $\NN$,
and which lies in $S_D$. Indeed, to obtain such a net, one selects one (arbitrary) point from the
intersection of $S_D$ with a ball of radius $2\a/D$ centered at each point from $\NN$.
This completes the proof.
\end{proof}

\begin{lemma}[Lower bound for a level set]          \label{for level set}
  There exist $c_1, c_2, \mu \in (0,1)$ such that the following holds.
  Let $\a = \mu \sqrt{N} \ge 1$ and $D \le c_1 \sqrt{N} e^{c_1 N / m}$.
  Then
  $$
  \P \big( \inf_{x \in S_D} \|Bx\|_2 < c_2 N/D \big) \le 2e^{-N}.
  $$
\end{lemma}

\begin{proof}
By Lemma~\ref{norm}, there exists $K \ge 1$ that depends only on the subgaussian moment
and such that
$$
\P (\|B\| > K \sqrt{N}) \le e^{-N}.
$$
Therefore, in order to complete the proof, it is enough to find
$\nu > 0$ which depends only on the subgaussian moment, and such that the event
$$
\EE :=  \Big\{ \inf_{x \in S_D} \|Bx\|_2 < \frac{\nu N}{2D}
  \text{ and } \|B\| \le K\sqrt{N} \Big\}
$$
has probability at most $e^{-N}$.

We claim that this holds with the following choice of parameters:
$$
\nu = \frac{1}{(3 C C_0)^2 e}, \quad
\mu = \frac{\nu}{9K}, \quad
c_1 = c \mu^2 \le \nu,
$$
where $C \ge 1$ and $c \in (0,1)$ are the constants from Lemma~\ref{for single vector}
and $C_0 \ge 1$ is the constant from Lemma~\ref{net}.

By choosing $\tilde{c}$ in the statement of Theorem~\ref{LCD random} suitably small, we
can assume that $N > \nu^{-2}$ (this is because by the assumptions, $N > m/\tilde{c} \ge 1/\tilde{c}$).
We apply Lemma~\ref{for single vector}
with $t = \nu \sqrt{N}/D$. Then recalling the choice of $\a$ and $c_1$ and our assumption on $D$,
one easily checks that the term $Ct$ dominates in the right hand side of \eqref{eq for single vector}:
$$
t \ge 1/D \quad \text{and} \quad t \ge e^{-c \a^2}.
$$
This gives for arbitrary $x_0 \in S_D$:
\[
  \P \Big( \|Bx_0\|_2 < \frac{ \nu N}{D}  \Big)
  \le \Big( \frac{3C\nu \sqrt{N}}{D} \Big)^{N-m}.
\]
Now we use Lemma~\ref{net}, which yields a small $(4\a/D)$-net $\NN$
of $S_D$. Taking the union bound, we get
\[
  p := \P \Big( \inf_{x_0 \in \NN} \|Bx_0\|_2 < \frac{ \nu N}{D} \Big)
  \le \Big( \frac{C_0 D}{\sqrt{N}} \Big)^N
  \Big( \frac{3C\nu \sqrt{N}}{D} \Big)^{N-m}.
\]
Denote $C_1 := 3 C C_0$. Using the fact that $c_1 \le \nu$ and our assumption on $D$, we have:
\begin{equation}                            \label{for the net}
  p
  \le C_1^N \Big( \frac{D}{\sqrt{N}} \Big)^m \nu^{N-m}
  \le C_1^N (\nu e^{\nu N / m})^m \nu^{N-m}
  \le C_1^{2N} \nu^N
  = e^{-N}.
\end{equation}

Assume $\EE$ occurs. Fix $x \in S_D$ for which $\|Bx\|_2 < \frac{\nu
N}{2D}$; it can be approximated by some element $x_0 \in \NN$ as
$$
\|x - x_0\|_2 \le \frac{4\mu \sqrt{N}}{D}.
$$
Therefore, by the triangle inequality we have
$$
  \|Bx_0\|_2 \le \|Bx\|_2 + \|B\| \cdot \|x-x_0\|_2
  \le \frac{\nu N}{2D} + K \sqrt{N} \cdot \frac{4\mu \sqrt{N}}{D}
  <  \frac{\nu N}{D},
$$
where in the last inequality we used our choice of $\mu$.

We have shown that the event $\EE$ implies the event that
$$
\inf_{x_0 \in \NN} \|Bx_0\|_2 < \frac{\nu N}{D},
$$
whose probability is at most $e^{-N}$ by \eqref{for the net}.
The proof is complete.
\end{proof}

\bigskip

\begin{proof}[Proof of Theorem~\ref{LCD random}.]
Consider $x \in S^{N-1}$ such that
$$
\LCD_{\a,c}(x) < c_1 \sqrt{N} e^{c_1 N / m},
$$
where $c_1$ is the constant from Lemma~\ref{for level set}.
Then, by the Definition~\ref{def level sets} of the level sets, either $x$ is compressible
or $x \in S_D$ for some $D \in \DD$, where
$$
\DD := \{ D :\; c_0 \sqrt{N} \le D < c_1 \sqrt{N} e^{c_1 N / m}, \; D = 2^k, \; k \in \N\}.
$$
Therefore, recalling the definition of the least common denominator of the subspace
$$
\LCD_{\a,c}(H^\perp) = \inf_{x \in S(H^\perp)} \LCD_{\a,c}(x),
$$
we can decompose the desired probability as follows:
\begin{align*}
p 
  &:= \P \big( \LCD_{\a,c}(H^\perp)  < c_1 \sqrt{N} e^{c_1 N / m} \big) \\
  &\le \P(H^\perp \cap \Comp \ne \emptyset) 
    + \sum_{D \in \DD}
      \P(H^\perp \cap S_D \ne \emptyset).
\end{align*}

By Lemma~\ref{random incompressible}, the first term in the right hand side
is bounded by $e^{-cN}$.
Further terms can be bonded using \eqref{navigation} and Lemma~\ref{for level set}:
$$
\P(H^\perp \cap S_D \ne \emptyset)
\le \P \big( \inf_{x \in S_D} \|Bx\|_2 = 0 \big)
\le 2e^{-N}.
$$
Since there are $|\DD| \le C'N$ terms in the sum, we conclude that
$$
p \le e^{-cN} + C'N e^{-N} \le e^{-c'N}.
$$
This completes the proof.
\end{proof}

\section{Decomposition of the sphere}                   \label{s: decomposition}

Now we begin the proof of Theorem~\ref{t: subgaussian}.
We will make several useful reductions first.

Without loss of generality, we can assume that the entries of $A$ have a
an absolutely continuous distribution. Indeed, we can add to each entry
an independent Gaussian random variable with small variance $\s$, and
later let $\s \to 0$.

Similarly,  we can assume that $n \ge n_0$, where $n_0$ is a
suitably large number that depends only on the subgaussian moment
$B$.

We let
$$
N = n  - 1+ d
$$
for some $d \ge 1$.
We can assume that
\begin{equation}                        \label{nd}
  1 \le d \le c_0 n,
\end{equation}
with suitably small constant $c_0 > 0$ that depends only on the
subgaussian moment $B$. Indeed, as we remarked in the Introduction,
for the values of $d$ above a constant proportion of $n$,
Theorem~\ref{t: subgaussian} follows from \eqref{constant
proportion}. Note that
$$
\sqrt{N} - \sqrt{n-1} \le \frac{d}{\sqrt{n}}.
$$

Using the decomposition of the sphere $S^{n-1} = \Comp \cup \Incomp$,
we break the invertibility problem into two subproblems,
for compressible and incompressible vectors:
\begin{multline}                        \label{two terms}
  \P \Big( s_n(A) \le \e \big (\sqrt{N} - \sqrt{n-1} \, \big) \Big)
  \le \P \big( s_n(A) \le \e \frac{d}{\sqrt{n}} \big) \\
  \le \P \big( \inf_{x \in \Comp(\d,\rho)} \|A x\|_2 \le \e \frac{d}{\sqrt{n}} \big)
    + \P \big( \inf_{x \in \Incomp(\d,\rho)} \|A x\|_2 \le \e \frac{d}{\sqrt{n}} \big).
\end{multline}

A bound for the compressible vectors follows from Lemma~\ref{l:
compressible}. Using \eqref{nd} we get
$$
\e \frac{d}{\sqrt{n}} \le c_0 \sqrt{n} \le c_0 \sqrt{N}.
$$
Hence, Lemma~\ref{l: compressible} implies
\begin{equation}\label{eq: compressible}
  \P \Big( \inf_{x \in \Comp(\d,\rho)} \|A x\|_2 \le\e \frac{d}{\sqrt{n}} \Big)
  \le e^{-c_3 N}.
\end{equation}

It remains to find a lower bound on $\|Ax\|$ for the incompressible vectors $x$.

\section{Invertibility via uniform distance bounds}        \label{s: via dist}

In this section, we reduce the problem of bounding $\|Ax\|_2$ for incompressible
vectors $x$ to the distance problem that we addressed in Section~\ref{s: distance}.

Let $X_1,\ldots,X_n \in \R^N$ denote the columns of the matrix $A$.
Given a subset $J \subseteq \{1,\ldots,n\}$ of cardinality $d$,
we consider the subspace
$$
H_J := \Span(X_k)_{k \in J} \subset \R^N.
$$

For levels $K_1, K_2 > 0$ that will only depend on $\d,\rho$, we
define the set of totally spread vectors
\begin{equation}                        \label{SJ}
  \Spread_J:= \Big\{ y \in S(\R^J) : \;
    \frac{K_1}{\sqrt{d}} \le |y_k| \le \frac{K_2}{\sqrt{d}}
    \quad \text{for all $k \in J$} \Big\}.
\end{equation}

In the following lemma, we let $J$ be a random subset uniformly distributed
over all subsets of $\{1,\ldots,n\}$ of cardinality $d$. To avoid
confusion, we often denote the probability and expectation over
the random set $J$ by $\P_J$ and $\E_J$, and with respect to the
random matrix $A$ by $\P_A$ and $\E_A$.

\begin{lemma}[Total spread]                     \label{total spread}
  For every $\d,\rho \in (0,1)$,
  there exist $K_1, K_2, c_0 > 0$ which depend only on $\d,\rho$,
  and such that the following holds.
  For every $x \in \Incomp(\d,\rho)$, the event
  $$
  \EE(x) : = \Big\{ \frac{P_J x}{\|P_J x\|_2} \in \Spread_J
  \quad \text{and} \quad
  \frac{\rho \sqrt{d}}{\sqrt{2n}} \le \norm{P_Jx}_2
  \le \frac{\sqrt{d}}{\sqrt{\d n}}
  \Big\}
  $$
  satisfies $\P_J(\EE(x)) > c_0^d$.
\end{lemma}

\begin{remark}
  The proof gives $K_1 = \rho \sqrt{\d/2}$, $K_2 = 1/K_1$, $c_0 = \rho^2 \d / 2e$.
  In the rest of the proof, we shall use definition \eqref{SJ} of
  $\Spread_J$ with these values of the levels $K_1$, $K_2$.
\end{remark}

\begin{proof}
Let $\s \subset \{1 \etc n\}$ be the subset from Lemma \ref{l:
spread}. Recall that the parameters $\d$ and $\rho$ depend only on
the subgaussian moment $B$ (see Lemma \ref{l: compressible}). By
choosing the constant $c_0$ in \eqref{nd} appropriately small, we
may assume that $d \le |\s|/2$.
 Then, using Stirling's approximation we have
\[
  \P_J (J \subset \s) = \binom{|\s|}{d} \Big / \binom{n}{d}
  > \Big( \frac{\rho^2 \d}{2 e} \Big)^d
  = c_0^d.
\]
If $J \subset \s$, then summing \eqref{eq: spread} over $k \in J$, we obtain
the required two-sided bound for $\|P_J x\|_2$.
This and \eqref{eq: spread} yields $\frac{P_J x}{\norm{P_J x}_2} \in \Spread_J$.
Hence $\EE(x)$ holds.
\end{proof}

\begin{lemma}[Invertibility via distance]                     \label{l: via dist}
  Let $\d,\rho \in (0,1)$.
  There exist $C_1, c_1 > 0$ which depend only on $\d,\rho$,
  and such that the following holds.
  Let $J$ be any $d$-element subset of $\{1 \etc n\}$.
  Then for every $\e > 0$
  \begin{equation}                                  \label{eq via dist}
    \P \Big( \inf_{x \in \Incomp(\d,\rho)} \|Ax\|_2 < c_1 \e \sqrt{\frac{d}{n}} \Big)
    \le C_1^d \cdot  \P \big( \inf_{z \in \Spread_J} \dist(Az, H_{J^c}) < \e \big).
  \end{equation}
\end{lemma}

\begin{remark}
  The proof gives $K_1 = \rho \sqrt{\d/2}$, $K_2 = 1/K_1$,
  $c_1 = \rho/\sqrt{2}$, $C_1 = 2e / \rho^2 \d$.
\end{remark}

\begin{proof}
Let $x \in \Incomp(\d, \rho)$.
 For every subset $J$ of $\{1,\ldots,n\}$ we have
$$
\|Ax\|_2 \ge \dist(Ax, H_{J^c})
= \dist(A P_J x, H_{J^c}).
$$
In case the event $\EE(x)$ of Lemma~\ref{total spread} holds,
we use the vector $z= \frac{P_J x}{\norm{P_J x}_2 } \in \Spread_J$ to check that
$$
\|Ax\|_2 \ge  \|P_J x\|_2 \, D(A,J),
$$
where the random variable
$$
D(A,J) = \inf_{z \in \Spread_J} \dist(Az, H_{J^c})
$$
is independent of $x$.
Moreover, using the estimate on $\|P_J x\|_2$ in the definition of the event $\EE(x)$,
we conclude that
\begin{equation}                        \label{EEJ implies}
  \EE(x) \quad \text{implies} \quad \|Ax\|_2 \ge c_1 \sqrt{\frac{d}{n}}  \, D(A,J).
\end{equation}

Define the event
\[
  \mathcal{F} := \big\{ A : \; \P_J(D(A,J) \ge \e)
  > 1 - c_0^d \big\},
\]
where $c_0$ is the constant from Lemma~\ref{total spread}.
Chebychev inequality and Fubini theorem then yield
$$
\P_A(\mathcal{F}^c) \le c_0^{-d} \E_A \P_J (D(A,J) < \e) = c_0^{-d}
\E_J \P_A (D(A,J) < \e).
$$
Since the entries of $A$ are independent and identically
distributed, the probability $\P_A (D(A,J) < \e)$ does not depend on $J$. Therefore,
the right hand side of the previous inequality coincides with the
right hand side of \eqref{eq via dist}.

Fix any realization of $A$ for which $\mathcal{F}$ occurs, and fix
any $x \in \Incomp(\d, \rho)$. Then
\[
  \P_J (D(A,J) \ge \e) +\P_J(\EE(x))
  > (1-c_0^d) + c_0^d = 1,
\]
 so we conclude that
\begin{equation}                                \label{two ineq for J}
  \P_J \big( \EE(x)
    \text{ and } D(A,J) \ge \e \big) > 0.
\end{equation}
We have proved that for every $x \in \Incomp(\d,\rho)$ there exists
a subset $J = J(x)$ that satisfies both $\EE(x)$ and $D(A,J) \ge
\e$. Using this $J$ in \eqref{EEJ implies}, we conclude that every
matrix $A$ for which the event $\mathcal{F}$ occurs satisfies
$$
\inf_{x \in \Incomp(\d,\rho)} \|Ax\|_2 \ge \e c_1 \sqrt{\frac{d}{n}}.
$$
This and the estimate of $\P_A(\mathcal{F}^c)$ completes the proof.
\end{proof}

\section{The uniform distance bound}                            \label{s: uniform distance}

In this section, we shall estimate the distance between a random ellipsoid
and a random independent subspace. This is the distance that we need to bound
in the right hand side of \eqref{eq via dist}.

Throughout this section, we let $J$ be a fixed subset of $\{1,\ldots,n\}$,
$|J| = d$. We shall use the notation introduced in the beginning of
Section~\ref{s: via dist}. Thus, $H_J$ denotes a random subspace, and $\Spread_J$
denotes the totally spread set whose levels $K_1$, $K_2$ depend only on
$\d$, $\rho$ in the definition of incompressibility.

We will denote by $K, K_0, C, c, C_1, c_1, \ldots$ positive numbers
that depend only on $\d$, $\rho$ and the subgaussian moment $B$.

\begin{theorem}[Uniform distance bound]             \label{uniform distance}
  For every $t > 0$,
  $$
  \P \Big( \inf_{z \in \Spread_J} \dist (Az, H_{J^c}) < t \sqrt{d} \Big)
  \le (Ct)^d + e^{-cN}.
  $$
\end{theorem}

Recall that $H_{J^c}$ is the span of $n-d$ independent random
vectors. Since their distribution is absolutely continuous (see the
beginning of Section~\ref{s: decomposition}), these vectors are
almost surely in general position, so
\begin{equation}                                            \label{dimH}
  \dim(H_{J^c}) = n-d.
\end{equation}

Without loss of generality, in the proof of Theorem~\ref{uniform distance}
we can assume that
\begin{equation}                                            \label{tzero}
  t \ge t_0 = e^{-\bar{c}N/d}
\end{equation}
with a suitably small $\bar{c} > 0$.

\subsection{First approach: nets and union bound}               \label{s: first}

We would like to prove Theorem~\ref{uniform distance} by a typical $\e$-net argument.
Theorem~\ref{distance} will give a useful probability bound for an individual $z \in S^{n-1}$.
We might then take a union bound over all $z$ in an $\e$-net of $\Spread_J$ and complete by
approximation. However, the standard approximation argument will leave us with a larger
error $e^{-cd}$ on the probability, which is unsatisfactory for small $d$.
To improve upon this step, we shall improve upon this approach using decoupling
in Section~\ref{s: second}.

For now, we start with a bound for an individual $z \in S^{n-1}$.

\begin{lemma}                       \label{nonuniform distance}
  Let $z \in S^{n-1}$ and $v \in \R^N$.
  Then for every $t$ that satisfies \eqref{tzero} we have
  $$
  \P \Big( \dist (Az, H_{J^c} + v) < t \sqrt{d} \Big)
  \le (C_1 t)^{2d-1}.
  $$
\end{lemma}

\begin{proof}
Denote the entries of matrix $A$ by $\xi_{ij}$.
Then the entries of the random vector $Az$,
$$
\zeta_i := (Az)_i = \sum_{j=1}^n \xi_{ij} z_j, \quad j = 1,\ldots, N,
$$
are independent and identically distributed mean zero random variables.
Moreover, since the random variables $\xi_{ij}$ are subgaussian
and $\sum_{j=1}^n z_j^2 = 1$, the random variables $\zeta_i$
are also subgaussian (see Fact~2.1 in \cite{LPRT}).

Therefore the random vector $X = Az$ and the random subspace $H = H_{J^c}$
satisfy the assumptions of Theorem~\ref{distance} with
$m = N - (n-d) = 2d - 1$
(we used \eqref{dimH} here).
An application of Theorem~\ref{distance} completes the proof.
\end{proof}

We will use this bound for every $z$ in an $\e$-net of $\Spread_J$.
To extend the bound to the whole set $\Spread_J$ by approximation, we need a
certain stability of the distance. This is easy to quantify and prove using
the following representation of the distance in matrix form.
Let $P$ be the orthogonal projection in $\R^N$ onto $(H_{J^c})^\perp$, and let
\begin{equation}                            \label{W}
  W := P A |_{\R^J}.
\end{equation}
Then for every $v \in \R^N$, the following identity holds:
\begin{equation}                            \label{dist via W}
  \dist(Az, H_{J^c} + v) = \|Wz - w\|_2,
  \quad \text{where } w = Pv.
\end{equation}

Since $|J| = d$ and almost surely $\dim(H_{J^c})^\perp = N - (n-d) =
2d-1$, the random matrix $W$ acts as an operator from a
$d$-dimensional subspace into a $(2d-1)$-dimensional subspace.
Although the entries of $W$ are not necessarily independent, we
expect $W$ to behave as if this was the case. To this end, we
condition on the realization of the subspace $(H_{J^c})$. Now
the operator $P$ becomes a fixed projection, and the
columns of $W$ become independent random vectors.
Then $W$ satisfies a version of Proposition~\ref{norm}:

\begin{proposition}                         \label{norm W}
  Let $P$ be an orthogonal projection in $\R^N$ of rank $d$
  and let $W= P A |_{\R^J}$ be a random matrix. Then
  $$
  \P \big( \|W\| > t \sqrt{d} \big)
  \le e^{-c_0 t^2 d}
  \qquad \text{for } t \ge C_0.
  $$
\end{proposition}

\begin{proof}
The argument is similar to that of Proposition~\ref{norm}. Let $\NN$
be a $(1/2)$-net of $S(\R^J)$ and $M$ be a $(1/2)$-net of $S(P
\R^N)$. Note that for $x \in \NN$, $y \in \MM$, we have $\< Wx,y\> =
\< Ax,y\> $. The proof is completed as in  Proposition~\ref{norm}.
\end{proof}

Using Proposition\ref{norm W}, we can choose a constant $K_0$ that depends
only on the subgaussian moment, and such that
\begin{equation}                            \label{W small}
  \P \big( \|W\| > K_0 \sqrt{d} \big)
  \le e^{-d}.
\end{equation}

With this bound on the norm of $W$, we can run the approximation argument
and prove the distance bound in Lemma~\ref{nonuniform distance} uniformly
over all $z \in \Spread_J$.

\begin{lemma}                       \label{uniform distance W small}
  Let $W$ be a random matrix as in Proposition \ref{norm W}.
  Then for every $t$ that satisfies \eqref{tzero} we have
  \begin{equation}                  \label{eq: uniform distance W small}
    \P \Big( \inf_{z \in \Spread_J} \|Wz\|_2 < t \sqrt{d}
    \text{ and } \|W\| \le K_0 \sqrt{d} \Big)
    \le (C_2 t)^d.
  \end{equation}
\end{lemma}

\begin{proof}
Let $\e = t/K_0$. By Proposition~\ref{nets}, there exists an $\e$-net $\NN$ of
$\Spread_J \subseteq S(\R^J)$ of cardinality
$$
|\NN| \le 2d \Big( 1 + \frac{2}{\e} \Big)^{d-1} \le 2 d \Big(
\frac{3K_0}{t} \Big)^{d-1}.
$$
Consider the event
$$
\EE := \Big\{ \inf_{z \in \NN} \|Wz\|_2 < 2 t \sqrt{d} \Big\}.
$$
Taking the union bound and using the representation \eqref{dist via W}
in Lemma~\ref{nonuniform distance}, we obtain
$$
\P(\EE) \le |\NN| \cdot \max_{z \in \NN} \P \big( \|Wz\|_2 \le
2t\sqrt{d} \big) \le 2 d \Big( \frac{3K_0}{t} \Big)^{d-1} (2 C_1
t)^{2d-1} \le (C_2 t)^d.
$$
Now, suppose the event in \eqref{eq: uniform distance W small} holds,
i.e. there exists $z' \in \Spread_J$ such that
$$
\|Wz'\|_2 < t \sqrt{d} \text{ and } \|W\| \le K_0 \sqrt{d}.
$$
Choose $z \in \NN$ such that $\|z-z'\|_2 \le \e$.
Then by the triangle inequality
$$
\|Wz\|_2 \le \|Wz'\|_2 + \|W\| \cdot \|z-z'\|_2
< t\sqrt{d} + K_0 \sqrt{d} \cdot \e
\le 2 t \sqrt{d}.
$$
Therefore, $\EE$ holds. The bound on the probability of $\EE$ completes the proof.
\end{proof}

Lemma~\ref{uniform distance W small} together with \eqref{W small} yield that
$$
\P \Big( \inf_{z \in \Spread_J} \|Wz\|_2 < t \sqrt{d} \Big)
\le (C_2 t)^d + e^{-d}.
$$
By representation \eqref{dist via W}, this is a weaker version of
Theorem~\ref{nonuniform distance}, with $e^{-d}$ instead of $e^{-cN}$.
Unfortunately, this bound is too weak for small $d$.
In particular, for square matrices we have $d=1$, and the bound is useless.

In the next section, we will refine our current approach using decoupling.

\subsection{Refinement: decoupling}               \label{s: second}

Our problem is that the probability bound in \eqref{W small} is too weak.
We will bypass this by decomposing our event according to all possible
values of $\|W\|$, and by decoupling the information about $\|Wz\|_2$
from the information about $\|W\|$.

\begin{proposition}[Decoupling]                 \label{decoupling}
  Let $W$ be an $N \times d$ matrix whose columns are independent
  random vectors. Let $\b > 0$ and
  let $z  \in S^{d-1}$ be a vector satisfying $|z_k| \ge
  \frac{\b}{\sqrt{d}}$ for all $k \in \{1 \etc d\}$.
  Then for every $0 < a < b$, we have
  $$
  \P \big( \|Wz\|_2 < a, \; \|W\| > b \big)
    \le 2 \sup_{x \in S^{d-1}, w \in \R^N}
      \P \Big( \|Wx-w\|_2 < \frac{\sqrt{2}}{\b} a \Big) \;
      \P \Big( \|W\| > \frac{b}{\sqrt{2}} \Big).
  $$
\end{proposition}

\begin{proof}
If $d=1$ then $\|W\| = \|Wz\|_2$, so the probability in the left hand side is zero.
So, let $d \ge 2$. Then we can decompose the index set $\{1,\ldots,n\}$
into two disjoint subsets $I$ and $H$ whose cardinalities differ by at most $1$,
say with $|I| = \lceil d/2 \rceil$.

We write $W = W_I + W_H$ where $W_I$ and $W_H$ are the submatrices
of $W$ with columns in $I$ and $H$ respectively. Similarly, for $z
\in \Spread_J$, we write $z = z_I + z_H$.

Since $\|W\|^2 \le \|W_I\|^2 + \|W_H\|^2$, we have
$$
\P \big( \|Wz\|_2 < a, \; \|W\| > b \big) = p_I + p_H,
$$
where
\begin{align*}
p_I &= \P \big( \|Wz\|_2 < a, \; \|W_H\| > b/\sqrt{2} \big) \\
    &= \P \big( \|Wz\|_2 < a \; \big| \; \|W_H\| > b/\sqrt{2} \big)
      \; \P \big( \|W_H\| > b/\sqrt{2} \big),
\end{align*}
and similarly for $p_H$. It suffices to bound $p_I$; the argument
for $p_H$ is similar.

Writing $Wz = W_I z_I + W_H z_H$ and using the independence of the
matrices $W_I$ and $W_H$, we conclude that
\begin{align}
p_I &\le \sup_{w \in \R^N} \P \big( \|W_I z_I - w\|_2 < a \big)
      \; \P \big( \|W_H\| > b/\sqrt{2} \big) \nonumber \\
    &\le \sup_{w \in \R^N} \P \big( \|W z_I - w\|_2 < a \big)
      \; \P \big( \|W\| > b/\sqrt{2} \big).             \label{decoupled}
\end{align}
(In the last line we used $W_I z_I = W z_I$ and $\|W_H\| \le \|W\|$).

By the assumption on $z$ and since $|I| \ge d/2$, we have
$$
\|z_I\|_2 = \Big( \sum_{k \in I} |z_k|^2 \Big)^{1/2} \ge \frac{\b}{\sqrt{2}}.
$$
Hence for $x := z_I / \|z_I\|_2$ and $u := w / \|z_I\|_2$, we obtain
$$
\P \big( \|W z_I - w\|_2 < a \big)
\le \P \big( \|W x - u\|_2 < \sqrt{2} a / \b \big).
$$
Together with \eqref{decoupled}, this completes the proof.
\end{proof}

We use this decoupling in the following refinement of Lemma~\ref{uniform distance W small}.

\begin{lemma}                       \label{uniform distance W}
  Let $W$ be a random matrix as in \eqref{W}, where $P$ is the orthogonal projection
  of $\R^N$ onto the random subspace $(H_{J^c})^\perp$, defined as
  in Theorem \ref{uniform distance}.
  Then for every $s \ge 1$ and every $t$ that satisfies \eqref{tzero}, we have
  \begin{align}                  \label{eq: uniform distance W}
    &\P \Big( \inf_{z \in \Spread_J} \|Wz\|_2 < t \sqrt{d}
    \text{ and } s K_0 \sqrt{d} < \|W\| \le 2 s K_0 \sqrt{d} \Big)
    \\
    &\le (C_3 t e^{-c_3 s^2})^d +e^{-cN}. \notag
  \end{align}
\end{lemma}

\begin{proof}
Let $\e = t / 2 s K_0$. By Proposition~\ref{nets}, there exists an $\e$-net $\NN$ of
$\Spread_J \subseteq S(\R^J)$ of cardinality
$$
|\NN| \le 2d \Big( 1 + \frac{2}{\e} \Big)^{d-1}
 \le 2 d \Big( \frac{6 s K_0}{t} \Big)^{d-1}.
$$
Consider the event
$$
\EE := \Big\{ \inf_{z \in \NN} \|Wz\|_2 < 2 t \sqrt{d}
  \text{ and } \|W\| > s K_0 \sqrt{d} \Big\}.
$$
We condition on the realization of the subspace $H_{J^c}$ as above
to make the columns of $W$ independent.
By the  definition \eqref{SJ} of $\Spread_J$, any $z \in \NN$ satisfies
the condition of the Decoupling Proposition~\ref{decoupling} with
$\b=K_1$. Taking the union bound and then using
Proposition~\ref{decoupling}, we obtain
\begin{align*}
   \P(\EE \mid H_{J^c})
   &\le |\NN| \cdot \max_{z \in \NN} \P \big( \|Wz\|_2 \le 2t\sqrt{d}
     \text{ and } \|W\| > s K_0 \sqrt{d} \mid H_{J^c} \big) \\
  &\le |\NN| \cdot 2 \max_{z \in S(\R^J), \; w \in \R^N}
    \P \Big( \|Wz- w\|_2 < \frac{\sqrt{2}}{K_1} \cdot 2t\sqrt{d} \mid H_{J^c}
    \Big) \\
    & \quad \cdot \P \Big( \|W\| > \frac{s K_0 \sqrt{d}}{\sqrt{2}} \mid H_{J^c}
    \Big).
\end{align*}
Assume now that $\LCD_{\a, c}(H_{J^c}^\perp) \ge c \sqrt{N}
e^{cN/m}$, where $\a$ and $c$ are as in Theorem~\ref{LCD random}.
Then using Proposition~\ref{norm W} and representation \eqref{dist
via W}, we conclude as in the proof of Theorem \ref{distance} that
$$
\P(\EE  \mid H_{J^c}) \le 4 d \Big( \frac{6 s K_0}{t} \Big)^{d-1}
  \cdot (C' t)^{2d-1} \cdot e^{-c's^2 d}
$$
 for any $t$ satisfying \eqref{tzero}.
Since $s \ge 1$ and $d \ge 1$, we can bound this as
$$
\P(\EE  \mid H_{J^c}) \le (C_3 t e^{-c_3 s^2})^d.
$$
 Therefore, by Theorem \ref{LCD random},
 \begin{align*}
   \P(\EE)
   &\le \P(\EE  \mid \LCD_{\a, c}(H_{J^c}^\perp) \ge c \sqrt{N}
            e^{cN/m} )
   + \P( \LCD_{\a, c}(H_{J^c}^\perp) < c \sqrt{N}  e^{cN/m} ) \\
   &\le (C_3 t e^{-c_3 s^2})^d + e^{-cN}.
 \end{align*}
Now, suppose the event in \eqref{eq: uniform distance W} holds, i.e.
there exists $z' \in \Spread_J$ such that
$$
\|Wz'\|_2 < t \sqrt{d} \text{ and } s K_0 \sqrt{d} < \|W\| \le 2 s K_0 \sqrt{d}.
$$
Choose $z \in \NN$ such that $\|z-z'\|_2 \le \e$.
Then by the triangle inequality
$$
\|Wz\|_2 \le \|Wz'\|_2 + \|W\| \cdot \|z-z'\|_2
< t\sqrt{d} + 2 s K_0 \sqrt{d} \cdot \e
\le 2 t \sqrt{d}.
$$
Therefore, $\EE$ holds. The bound on the probability of $\EE$ completes the proof.
\end{proof}

\bigskip

\begin{proof}[Proof of the Uniform Distance Theorem~\ref{uniform distance}]
Recall that, without loss of generality, we assumed that \eqref{tzero} held.
 Let $k_1$ be the smallest natural number such that
 \begin{equation}  \label{k_1}
   2^{k_1} \cdot K_0 \sqrt{d} > C_0 \sqrt{N},
 \end{equation}
 where $C_0$ and $K_0$ are constants from Lemma~\ref{norm} and
Lemma~\ref{uniform distance W} respectively.
Summing the probability estimates of Proposition~\ref{uniform distance W small} 
and Lemma~\ref{uniform distance W} for $s = 2^k$, $k=1,
\ldots, k_1$, we conclude that
\begin{align*}
  &\P \Big( \inf_{z \in \Spread_J} \|Wz\|_2 < t \sqrt{d} \Big) \\
  &\le (C_2 t)^d
   + \sum_{s = 2^k, \; k=1, \ldots, k_1}
       \Big ( (C_3 t e^{-c_3 s^2})^d+ e^{-cN} \Big )
   + \P(\norm{W}>C_0 \sqrt{N})  \\
  &\le (C_4 t)^d + k_1 e^{-cN} + \P(\norm{A}>C_0 \sqrt{N}).
\end{align*}
 By \eqref{k_1} and Proposition~\ref{norm}, the last expression does not
 exceed $(C t)^d +  e^{-cN}$.
In view of representation \eqref{dist via W}, this completes the
proof.
\end{proof}

\section{Completion of the proof}               \label{s: completion}

In Section~\ref{s: via dist}, we reduced the invertibility problem
for incompressible vectors to computing the distance between a random ellipsoid
and a random subspace. This distance was estimated in Section~\ref{s: uniform distance}.
These together lead to the following invertibility bound:

\begin{theorem}[Invertibility for incompressible vectors]  \label{t: incompressible}
  Let $\d,\rho \in (0,1)$.
  There exist $C, c > 0$ which depend only on $\d,\rho$,
  and such that the following holds.
  For every $t > 0$,
  $$
  \P \Big( \inf_{x \in \Incomp(\d,\rho)} \|Ax\|_2 < t \frac{d}{\sqrt{n}} \Big)
    \le (Ct)^{d} + e^{-cN}.
  $$
\end{theorem}

\begin{proof}
Without loss of generality, we can assume that \eqref{tzero} holds.
We use Lemma~\ref{l: via dist} with $\e = t\sqrt{d}$ and then
Theorem~\ref{uniform distance} to get the bound $(C't)^d$ on the
desired probability. This completes the proof.
\end{proof}

\begin{proof} [Proof of Theorem \ref{t: subgaussian}]
  This follows directly from \eqref{two terms},  \eqref{eq: compressible},
  and Theorem~\ref{t: incompressible}.
\end{proof}

\end{document}